\documentclass[11pt,leqno]{article}
\usepackage[margin=1in]{geometry} 
\usepackage{slantsc}

\usepackage{amssymb,amsfonts,amsmath,bbm,mathrsfs,stmaryrd,mathtools}
\usepackage{xcolor}
\usepackage{url}

\usepackage{extarrows}

\usepackage[shortlabels]{enumitem}
\usepackage{tensor}

\usepackage{xr}

\usepackage[T1]{fontenc}

\usepackage[colorlinks,
linkcolor=black!75!red,
citecolor=blue,
pdftitle={},
pdfproducer={pdfLaTeX},
bookmarksopen=true,
bookmarksnumbered=true,
backref=page]{hyperref}

\usepackage{tikz}
\usetikzlibrary{arrows,calc,decorations.pathreplacing,decorations.markings,intersections,shapes.geometric,through,fit,shapes.symbols,positioning,decorations.pathmorphing}

\pgfdeclarelayer{edgelayer}
\pgfdeclarelayer{nodelayer}
\pgfsetlayers{edgelayer,nodelayer,main}
\tikzstyle{none}=[inner sep=1pt]
\tikzstyle{None}=[inner sep=1pt, fill=white]
\tikzstyle{dashedcircle}=[circle, draw=gray, dashed, inner sep=6pt]
\tikzstyle{box}=[draw=black, fill=white, inner sep=.5ex, rounded corners=.1ex]
\tikzstyle{roundedbox}=[draw=black, fill=white, inner sep=.5ex, rounded corners=1ex]
\tikzstyle{cross}=[preaction={draw=white, -, line width=3pt}]
\tikzstyle{arrow}=[postaction=decorate]
\newcommand{\markat}{0.5}
\newcommand{\markwithsym}{>}
\newcommand{\markwith}{{\arrow[black]{\markwithsym}}}
\pgfkeys{/tikz/.cd, markat/.store in=\markat, markwith/.store in=\markwithsym} 
\tikzset{decoration={markings, mark=at position \markat with \markwith}}

\newcommand\xrsquigarrow[1]{%
  \mathrel{%
    \begin{tikzpicture}[baseline= {( $ (current bounding box.south) + (0,-0.5ex) $ )}]
      \node[inner sep=.5ex] (a) {$\scriptstyle #1$};
      \path[draw,
      <-,
      >=stealth,
      decorate,
      decoration={zigzag,amplitude=0.7pt,segment length=1.2mm,pre=lineto,pre length=4pt}] 
      (a.south east) -- (a.south west);
    \end{tikzpicture}}%
}

\usepackage{braket}

\usepackage[amsmath,thmmarks,hyperref]{ntheorem}
\usepackage{cleveref}

\newcommand\nthalias[1]{\AddToHook{env/#1/begin}{\crefalias{lemma}{#1}}}

\nthalias{definition}
\nthalias{example}
\nthalias{examples}
\nthalias{remark}
\nthalias{remarks}
\nthalias{convention}
\nthalias{notation}
\nthalias{construction}
\nthalias{theoremN}
\nthalias{propositionN}
\nthalias{corollaryN}
\nthalias{lemma}
\nthalias{proposition}
\nthalias{corollary}
\nthalias{theorem}
\nthalias{conjecture}
\nthalias{question}
\nthalias{assumption}

\creflabelformat{enumi}{#2#1#3}

\crefname{section}{Section}{Sections}
\crefformat{section}{#2Section~#1#3} 
\Crefformat{section}{#2Section~#1#3} 

\crefname{subsection}{\S}{\S\S}
\AtBeginDocument{%
  \crefformat{subsection}{#2\S#1#3}%
  \Crefformat{subsection}{#2\S#1#3}%
}

\crefname{subsubsection}{\S}{\S\S}
\AtBeginDocument{%
  \crefformat{subsubsection}{#2\S#1#3}%
  \Crefformat{subsubsection}{#2\S#1#3}%
}

\theoremstyle{plain}

\newtheorem{lemma}{Lemma}[section]
\newtheorem{proposition}[lemma]{Proposition}

\newtheorem{theorem}[lemma]{Theorem}

\theoremstyle{plain}
\theoremnumbering{Alph}

\theoremstyle{plain}
\theorembodyfont{\upshape}
\theoremsymbol{\ensuremath{\blacklozenge}}

\newtheorem{remark}[lemma]{Remark}
\newtheorem{remarks}[lemma]{Remarks}

\newtheorem{notation}[lemma]{Notation}

\crefname{definition}{definition}{definitions}
\crefformat{definition}{#2definition~#1#3} 
\Crefformat{definition}{#2Definition~#1#3} 

\crefname{ex}{example}{examples}
\crefformat{example}{#2example~#1#3} 
\Crefformat{example}{#2Example~#1#3} 

\crefname{exs}{example}{examples}
\crefformat{examples}{#2example~#1#3} 
\Crefformat{examples}{#2Example~#1#3} 

\crefname{remark}{remark}{remarks}
\crefformat{remark}{#2remark~#1#3} 
\Crefformat{remark}{#2Remark~#1#3} 

\crefname{remarks}{remark}{remarks}
\crefformat{remarks}{#2remark~#1#3} 
\Crefformat{remarks}{#2Remark~#1#3} 

\crefname{convention}{convention}{conventions}
\crefformat{convention}{#2convention~#1#3} 
\Crefformat{convention}{#2Convention~#1#3} 

\crefname{notation}{notation}{notations}
\crefformat{notation}{#2notation~#1#3} 
\Crefformat{notation}{#2Notation~#1#3} 

\crefname{table}{table}{tables}
\crefformat{table}{#2table~#1#3} 
\Crefformat{table}{#2Table~#1#3} 

\crefname{lemma}{lemma}{lemmas}
\crefformat{lemma}{#2lemma~#1#3} 
\Crefformat{lemma}{#2Lemma~#1#3} 

\crefname{proposition}{proposition}{propositions}
\crefformat{proposition}{#2proposition~#1#3} 
\Crefformat{proposition}{#2Proposition~#1#3} 

\crefname{propositionN}{proposition}{propositions}
\crefformat{propositionN}{#2proposition~#1#3} 
\Crefformat{propositionN}{#2Proposition~#1#3} 

\crefname{corollary}{corollary}{corollaries}
\crefformat{corollary}{#2corollary~#1#3} 
\Crefformat{corollary}{#2Corollary~#1#3} 

\crefname{corollaryN}{corollary}{corollaries}
\crefformat{corollaryN}{#2corollary~#1#3} 
\Crefformat{corollaryN}{#2Corollary~#1#3} 

\crefname{theorem}{theorem}{theorems}
\crefformat{theorem}{#2theorem~#1#3} 
\Crefformat{theorem}{#2Theorem~#1#3} 

\crefname{theoremN}{theorem}{theorems}
\crefformat{theoremN}{#2theorem~#1#3} 
\Crefformat{theoremN}{#2Theorem~#1#3} 

\crefname{enumi}{}{}
\crefformat{enumi}{#2#1#3}
\Crefformat{enumi}{#2#1#3}

\crefname{assumption}{assumption}{Assumptions}
\crefformat{assumption}{#2assumption~#1#3} 
\Crefformat{assumption}{#2Assumption~#1#3} 

\crefname{construction}{construction}{Constructions}
\crefformat{construction}{#2construction~#1#3} 
\Crefformat{construction}{#2Construction~#1#3} 

\crefname{question}{question}{Questions}
\crefformat{question}{#2question~#1#3} 
\Crefformat{question}{#2Question~#1#3} 

\crefname{equation}{}{}
\crefformat{equation}{(#2#1#3)} 
\Crefformat{equation}{(#2#1#3)}

\numberwithin{equation}{section}
\renewcommand{\theequation}{\thesection-\arabic{equation}}

\theoremstyle{nonumberplain}
\theoremsymbol{\ensuremath{\blacksquare}}

\newtheorem{proof}{Proof}
\newcommand\pf[1]{\newtheorem{#1}{Proof of \Cref{#1}}}

\newcommand\bC{{\mathbb C}}

\newcommand\cR{{\mathcal R}}

\newcommand\cV{{\mathcal V}}
\newcommand\cW{{\mathcal W}}
\newcommand\cX{{\mathcal X}}
\newcommand\cY{{\mathcal Y}}
\newcommand\cZ{{\mathcal Z}}

\newcommand\wt{\widetilde}

\DeclareMathOperator{\End}{\mathrm{End}}
\DeclareMathOperator{\Hom}{\mathrm{Hom}}

\DeclareMathOperator{\im}{\mathrm{im}}

\DeclareMathOperator{\tr}{tr}
\DeclareMathOperator{\At}{\mathrm{At}}
\DeclareMathOperator{\ev}{\mathrm{ev}}
\DeclareMathOperator{\db}{\mathrm{db}}
\DeclareMathOperator{\Ann}{\mathrm{Ann}}

\newcommand{\id}{\mathrm{id}}

\newcommand\numberthis{\addtocounter{equation}{1}\tag{\theequation}}

\newcommand{\cat}[1]{\textsc{#1}}

\newcommand{\qedhere}{\mbox{}\hfill\ensuremath{\blacksquare}}

\title{A new characterization of Kac-type discrete quantum groups}
\author{Alexandru Chirvasitu and Andre Kornell}

\renewcommand{\:}{\colon}
\newcommand{\iso}{\cong}

\newcommand{\CC}{\mathbb C}
\renewcommand{\tensor}{\mathbin{\otimes}}
\newcommand{\stensor}{\mathbin{\overline \otimes}}

\begin{document}

\date{}

\newcommand{\Addresses}{{
  \bigskip
  \footnotesize

  \textsc{Department of Mathematics, University at Buffalo}
  \par\nopagebreak
  \textsc{Buffalo, NY 14260-2900, USA}  
  \par\nopagebreak
  \textit{E-mail address}: \texttt{achirvas@buffalo.edu}

  \medskip
  
  \textsc{Department of Mathematical Sciences, New Mexico State University}
  \par\nopagebreak
  \textsc{Las Cruces, NM 88003, USA}
  \par\nopagebreak
  \textit{E-mail address}: \texttt{kornell@nmsu.edu}

}}

\maketitle

\begin{abstract}
We obtain two related characterizations of discrete quantum groups and discrete quantum groups of Kac type as allegorical group objects in the symmetric monoidal dagger category of quantum sets and relations, of interest to quantum predicate logic and quantum information theory. Specifically, we characterize discrete quantum groups by the existence of an inversion relation and discrete quantum groups of Kac type by the existence of an inversion function. This confirms a conjectured description of discrete quantum groups of Kac type and brings them within the purview of category-internal universal algebra. 
\end{abstract}

\noindent {\em Key words:
  Kac type;
  antipode;
  diagonal state;
  discrete quantum group;
  hereditarily atomic;
  quantum function;
  quantum relation;
  quantum set
}

\vspace{.5cm}

\noindent{MSC 2020: 46L67; 20G42; 46L89; 81P10; 46L30; 18D15; 18M10; 18M05

  
}


\section*{Introduction}

Discrete quantum groups emerged from the study of compact quantum groups \cite[section~3]{zbMATH04152742}. The definition that is now standard was given by Van Daele \cite[Definition~2.3]{zbMATH00869786}. He defined a \emph{discrete quantum group} to be a multiplier Hopf $*$-algebra $(A, \Delta)$ in the sense of \cite[Definition~2.4]{zbMATH00569708} such that $A$ is a direct sum of full matrix algebras. Vaes recently established another characterization of discrete quantum groups \cite{362309, zbMATH07828323}. To state this characterization succinctly, we recall a few basic notions.

First, a von Neumann algebra is
\begin{itemize}[wide]
\item \emph{atomic} \cite[\S 10.21]{strat} if the supremum of the minimal projections is the unit $1$,
\item \emph{hereditarily atomic} \cite[Definition 5.3]{zbMATH07287276} if every von Neumann subalgebra is atomic. 
\end{itemize}
The latter class coincides \cite[Proposition 5.4]{zbMATH07287276} with the class of categorical von Neumann products
\begin{equation}\label{eq:prod.dec}
  \begin{aligned}
    \prod^{W^*}_{i\in I}M_{n_i}
    &:=
      \left\{(x_i)_{i\in I}\ :\ x_i\in M_{n_i},\,\sup_i\|x_i\|<\infty\right\}\\
    M_{n_i}
    &:=\text{matrix algebra }M_{n_i}(\bC)
  \end{aligned}
\end{equation}
(this is the \emph{$\ell^{\infty}$-sum} of the factors $M_{n_i}$, in the language of \cite[\S 1.1, p.63]{hlmsk_fa}). The primary motivation for considering hereditarily atomic von Neumann algebras is that they form a quantum generalization of discrete quantum spaces, i.e, sets \cite[\S~I]{zbMATH07287276}.

Hereditarily atomic von Neumann algebras may be characterized by the non-degeneracy of their diagonal projections. For each von Neumann algebra $M$,
\begin{itemize}
\item the \emph{diagonal projection} is the maximum projection $\delta_M \in M \stensor M^{op}$ that is orthogonal to $p \tensor (1-p)$ for each projection $p \in M$,
\item a \emph{diagonal state} is a normal state $\mu\:M \stensor M^{op} \to \CC$ such that $\mu(p \tensor (1-p)) = 0$ for each projection $p \in M$.
\end{itemize}
In other words, a diagonal state is any normal state that is supported on the diagonal projection. It can be shown that $M$ is hereditarily atomic iff its diagonal projection $\delta_M$ is non-degenerate in the sense that it is not orthogonal to $p \tensor p$ for any nonzero projection $p \in M$ \cite[Proposition~A.1.2]{zbMATH07287276}.

We may define a \emph{discrete quantum monoid} to be a hereditarily atomic von Neumann algebra $M$ that is equipped with $W^*$-morphisms, i.e., unital normal $*$-homomorphisms
  \begin{equation*}
    M\xrightarrow{\quad\Delta\quad}
    M\stensor M
    \quad\text{and}\quad
    M\xrightarrow{\quad\varepsilon\quad}\bC
  \end{equation*}
such that
  \begin{equation*}
    (\Delta \stensor \id)\circ \Delta = (\id \stensor \Delta) \circ \Delta
    ,\quad
    (\varepsilon \stensor \id) \circ \Delta = \id,
    \quad\text{and}\quad
    (\id \stensor \varepsilon) \circ \Delta = \id.
  \end{equation*}
We may then paraphrase Vaes's results as follows.
\begin{theorem}[Vaes]\label{thm:Vaes}
A discrete quantum monoid $(M, \Delta, \varepsilon)$ is a discrete quantum group iff
\begin{enumerate}[(1),wide]
\item\label{item:thm:Vaes:mul} for every normal state $\mu$ on $M$, there exists a normal state $\nu$ on $M$ with $(\mu \stensor \nu)(\Delta(e)) > 0$,
\item\label{item:thm:Vaes:mur} for every normal state $\nu$ on $M$, there exists a normal state $\mu$ on $M$ with $(\mu \stensor \nu)(\Delta(e)) > 0$,
\end{enumerate}
where $e$ is the support projection of $\varepsilon$.
\end{theorem}

Some remarks are necessary to interpret this result in the context of Van Daele's definition of a discrete quantum group. When we say that $(M, \Delta, \varepsilon)$ is a discrete quantum group, we mean that $(M_{00}, \Delta)$ is a multiplier Hopf $*$-algebra and that $\varepsilon$ is its unique counit. Here, $M_{00}$ is the $*$-subalgebra of $M$ that is generated by its minimal projections. If $M$ is the von Neumann product $\Pi_{i \in I}^{W^*} M_{n_i}$, then $M_{00}$ is the algebraic direct sum $\bigoplus_{i \in I} M_{n_i}$, and its multiplier algebra is the algebraic product $\mathrm{Mult}(M_{00}) = \Pi_{i \in I} M_{n_i}$. Hence, $M \subseteq \mathrm{Mult}(M_{00})$, and similarly, $M \stensor M \subseteq \mathrm{Mult}(M_{00} \tensor M_{00})$. Thus, $\Delta$ and $\varepsilon$ restrict to maps
\begin{equation*}
  \begin{aligned}
    M_{00}
    &\xrightarrow[\quad\text{$*$-morphism}\quad]{\quad\Delta\quad}
      \mathrm{Mult}(M_{00}\otimes M_{00}),\\
    M_{00}
    &\xrightarrow[\text{$*$-morphism}]{\quad\varepsilon\quad}
      \bC,
  \end{aligned}  
\end{equation*}
and \Cref{thm:Vaes} characterizes those discrete quantum monoids $(M, \Delta, \varepsilon)$ such that $(M_{00}, \Delta)$ is a multiplier Hopf $*$-algebra.

Conversely, every multiplier Hopf $*$-algebra $(A, \Delta)$ with $A$ being a direct sum of full matrix algebras, i.e., every discrete quantum group in the sense of Van Daele, arises in this way. Indeed, in this case, the bounded elements of the multiplier $*$-algebra $\mathrm{Mult}(A)$ form a hereditarily atomic von Neumann algebra $M$, and $M_{00} \iso A$. Furthermore, the comultiplication $\Delta$ and its unique counit $\varepsilon$ have ranges in $M \stensor M$ and $\CC$, respectively, by spectral permanence and lift uniquely to $W^*$-morphisms $M \to M \stensor M$ and $M \to \CC$, respectively, because they are non-degenerate and $M$ is the enveloping von Neumann algebra of the $C^*$-completion of $A$. Thus, every discrete quantum group in the sense of Van Daele may be regarded as a discrete quantum monoid, and \Cref{thm:Vaes} provides an alternative characterization of this class.

The present article provides a similar characterization of discrete quantum groups of Kac type.

\begin{theorem}\label{thm:us}
  A discrete quantum monoid $(M, \Delta, \varepsilon)$ is a discrete quantum group of Kac type iff there exists a $W^*$-morphism
  \begin{equation*}
    M\xrightarrow{\quad s\quad}M^{op}
  \end{equation*}
  such that
  \begin{enumerate}[(1),wide]
  \item\label{item:thm:us:l} $\varphi \circ (s \stensor \id) \circ \Delta = \varepsilon$ for all diagonal states $\varphi\: M^{op} \stensor M \to \CC$,
  \item\label{item:thm:us:r} $\varphi \circ (\id \stensor s) \circ \Delta = \varepsilon$ for all diagonal states $\varphi\: M \stensor M^{op} \to \CC$.
  \end{enumerate}
\end{theorem}
\noindent We say that a discrete quantum group is of \emph{Kac type} if its antipode \cite[\S 4]{zbMATH00869786} is a $*$-map. Such discrete quantum groups are dual to compact quantum groups of Kac type, which are typically defined \cite[post Proposition 1.7.9]{NeTu13} by the condition that their Haar states are tracial.

Conditions \Cref{item:thm:us:l} and \Cref{item:thm:us:r} of \Cref{thm:us} entail those of \Cref{thm:Vaes} (see \Cref{le:supp.proj.dom} below), so the latter renders \Cref{thm:us} equivalent to the following ostensibly weaker version.

\begin{theorem}\label{th:is.kac.bis}
  A discrete quantum group $(M, \Delta, \varepsilon)$ is of Kac type iff there exists a $W^*$-morphism $s\: M \to M^{op}$ such that
  such that
  \begin{enumerate}[(1),wide]
  \item\label{item:th:is.kac.bis:l} $\varphi \circ (s \stensor \id) \circ \Delta = \varepsilon$ for all diagonal states $\varphi\: M^{op} \stensor M \to \CC$,
  \item\label{item:th:is.kac.bis:r} $\varphi \circ (\id \stensor s) \circ \Delta = \varepsilon$ for all diagonal states $\varphi\: M \stensor M^{op} \to \CC$.
  \end{enumerate}
\end{theorem}

By way of some additional motivation beyond whatever intrinsic interest they may possess, we briefly describe the considerations that led to the statements of \Cref{thm:Vaes,thm:us}. The category of all hereditarily atomic von Neumann algebras and \emph{quantum relations} in the sense of Weaver \cite[Definition 2.1]{zbMATH06008057} behaves much like an \emph{allegory} \cite[\S 2.11]{zbMATH00045228}. Formally, it is a \emph{dagger compact quantaloid} \cite[Definition 3.1]{2504.18266v1}, and as such, it is equivalent to the category $\cat{qRel}$ of \emph{quantum sets and (binary) relations} introduced in \cite[\S III]{zbMATH07287276}. Various classes of discrete quantum structures may be defined by internalizing the allegorical definitions of their classical counterparts \cite{zbMATH07828323}. \Cref{thm:Vaes,thm:us} demonstrate that internalizing two distinct allegorical definitions of a group yields the classes of all discrete quantum groups and of discrete quantum groups of Kac type.

The dagger compact quantaloid $\cat{qRel}$ is in particular a dagger compact category, i.e. a \emph{strongly compact category} \cite[Definition 12]{MR2724650}, so it interprets the string diagrams pervasive throughout \cite{MR2724650}. Up to equivalence of dagger compact quantaloids, a discrete quantum monoid may be defined to be an object in $\cat{qRel}$ together with morphisms
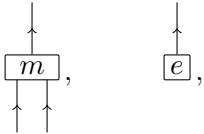
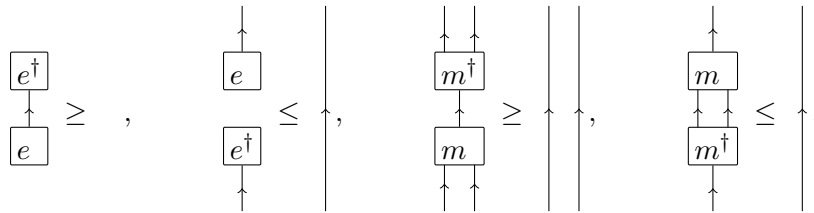
\begin{equation*}
  \begin{aligned}
    \begin{tikzpicture}
      \begin{pgfonlayer}{nodelayer}
	\node [style=box] (m) at (0,-0.5) {\;$m$\;};
	\node (md) at (0, 0.5) {};
	\node (a1) at (-0.2,-0.5) {};
	\node (b1) at (0.2,-0.5) {};
	\node (X1) at (-0.2,-1.5) { };
	\node (Y1) at (0.2,-1.5) { };
      \end{pgfonlayer}
      \begin{pgfonlayer}{edgelayer}
	\draw [arrow] (m) to (md);
	\draw [arrow] (X1) to (a1);
	\draw [arrow] (Y1) to (b1);
      \end{pgfonlayer}
    \end{tikzpicture}
  \end{aligned}\,,
  \qquad\quad
  \begin{aligned}
    \begin{tikzpicture}
      \begin{pgfonlayer}{nodelayer}
        \node (cloud) at (0, -1.5){};
	\node [style=box] (e) at (0,-0.5) {$e$};
	\node (ed) at (0, 0.5) {};
      \end{pgfonlayer}
      \begin{pgfonlayer}{edgelayer}
	\draw [arrow] (e) to (ed);
      \end{pgfonlayer}
    \end{tikzpicture}
  \end{aligned}\,,
\end{equation*}
such that
\begin{equation*}
\begin{aligned}
\begin{tikzpicture}
\begin{pgfonlayer}{nodelayer}
	\node [style=box] (e) at (0,-0.5) {$e^{\phantom{\dagger}}$};
	\node [style=box] (ed) at (0, 0.5) {$e^\dagger$};
\end{pgfonlayer}
\begin{pgfonlayer}{edgelayer}
	\draw [arrow] (e) to (ed);
\end{pgfonlayer}
\end{tikzpicture}
\end{aligned}
\; \geq \;
\quad,
\qquad\quad
\begin{aligned}
\begin{tikzpicture}
\begin{pgfonlayer}{nodelayer}
	\node [style=box] (e) at (0,0.5) {$e^{\phantom{\dagger}}$};
	\node [style=box] (ed) at (0, -0.5) {$e^\dagger$};
	\node (out) at (0, 1.5) { };
	\node (in) at (0, -1.5) { };
\end{pgfonlayer}
\begin{pgfonlayer}{edgelayer}
	\draw [arrow] (in) to (ed);
	\draw [arrow] (e) to (out);
\end{pgfonlayer}
\end{tikzpicture}
\end{aligned}
\;\leq \;
\begin{aligned}
\begin{tikzpicture}
\begin{pgfonlayer}{nodelayer}
	\node (out) at (0, 1.5) { };
	\node (in) at (0, -1.5) { };
\end{pgfonlayer}
\begin{pgfonlayer}{edgelayer}
	\draw [arrow] (in) to (out);
\end{pgfonlayer}
\end{tikzpicture}
\end{aligned},
\qquad\quad
\begin{aligned}
\begin{tikzpicture}
\begin{pgfonlayer}{nodelayer}
	\node [style=box] (m) at (0,-0.5) {$m^{\phantom{\dagger}}$};
	\node [style=box] (md) at (0, 0.5) {$m^\dagger$};
	\node (a1) at (-0.2,-0.5) {};
	\node (a2) at (-0.2,0.5) {};
	\node (b1) at (0.2,-0.5) {};
	\node (b2) at (0.2,0.5) {};
	\node (X1) at (-0.2,-1.5) { };
	\node (X2) at (-0.2,1.5) { };
	\node (Y1) at (0.2,-1.5) { };
	\node (Y2) at (0.2,1.5) { };
\end{pgfonlayer}
\begin{pgfonlayer}{edgelayer}
	\draw [arrow] (m) to (md);
	\draw [arrow] (X1) to (a1);
	\draw [arrow] (a2) to (X2);
	\draw [arrow] (Y1) to (b1);
	\draw [arrow] (b2) to (Y2);
\end{pgfonlayer}
\end{tikzpicture}
\end{aligned}
\; \geq \;
\begin{aligned}
\begin{tikzpicture}
\begin{pgfonlayer}{nodelayer}
	\node (X1) at (-0.2,-1.5) { };
	\node (X2) at (-0.2,1.5) { };
	\node (Y1) at (0.2,-1.5) { };
	\node (Y2) at (0.2,1.5) { };
\end{pgfonlayer}
\begin{pgfonlayer}{edgelayer}
	\draw [arrow] (X1) to (X2);
	\draw [arrow] (Y1) to (Y2);
\end{pgfonlayer}
\end{tikzpicture}
\end{aligned},
\qquad\quad
\begin{aligned}
\begin{tikzpicture}
\begin{pgfonlayer}{nodelayer}
	\node [style=box] (m) at (0,0.5) {$m^{\phantom{\dagger}}$};
	\node [style=box] (md) at (0, -0.5) {$m^\dagger$};
	\node (out) at (0, 1.5) { };
	\node (in) at (0, -1.5) { };
	\node (a1) at (-0.2,-0.5) {};
	\node (a2) at (-0.2,0.5) {};
	\node (b1) at (0.2,-0.5) {};
	\node (b2) at (0.2,0.5) {};
\end{pgfonlayer}
\begin{pgfonlayer}{edgelayer}
	\draw [arrow] (in) to (md);
	\draw [arrow] (m) to (out);
	\draw [arrow] (a1) to (a2);
	\draw [arrow] (b1) to (b2);
\end{pgfonlayer}
\end{tikzpicture}
\end{aligned}
\;\leq \;
\begin{aligned}
\begin{tikzpicture}
\begin{pgfonlayer}{nodelayer}
	\node (out) at (0, 1.5) { };
	\node (in) at (0, -1.5) { };
\end{pgfonlayer}
\begin{pgfonlayer}{edgelayer}
	\draw [arrow] (in) to (out);
\end{pgfonlayer}
\end{tikzpicture}
\end{aligned},
\end{equation*}
expressing that $e$ and $m$ are \emph{maps}, which are also known as \emph{functions} \cite[Definition 4.1]{zbMATH07287276}, and
\begin{equation*}
  \begin{aligned}
    \begin{tikzpicture}
      \begin{pgfonlayer}{nodelayer}
	\node [style=box] (m1) at (0,-0.5) {\;$m$\;};
	\node [style=box] (m2) at (0.4, 0.5) {\;$m$\;};
	\node (a1) at (-0.2,-0.5) {};
	\node (a2) at (0.2,0.5) {};
	\node (b1) at (0.2,-0.5) {};
	\node (b2) at (0.6,0.5) {};
	\node (X) at (-0.2,-1.5) { };
	\node (Y) at (0.2,-1.5) { };
	\node (Z) at (0.6,-1.5) { };
	\node (out) at (0.4,1.5) { };
      \end{pgfonlayer}
      \begin{pgfonlayer}{edgelayer}
	\draw [arrow] (X) to (a1);
	\draw [arrow] (Y) to (b1);
	\draw [arrow] (Z) to (b2);
	\draw [arrow, out = 90, in = 270] (m1) to (a2);
	\draw [arrow, out = 90, in = 270] (m2) to (out);
      \end{pgfonlayer}
    \end{tikzpicture}
  \end{aligned}
  \; = \;
  \begin{aligned}
    \begin{tikzpicture}
      \begin{pgfonlayer}{nodelayer}
	\node [style=box] (m1) at (0,-0.5) {\;$m$\;};
	\node [style=box] (m2) at (-0.4, 0.5) {\;$m$\;};
	\node (a1) at (0.2,-0.5) {};
	\node (a2) at (-0.2,0.5) {};
	\node (b1) at (-0.2,-0.5) {};
	\node (b2) at (-0.6,0.5) {};
	\node (X) at (0.2,-1.5) { };
	\node (Y) at (-0.2,-1.5) { };
	\node (Z) at (-0.6,-1.5) { };
	\node (out) at (-0.4,1.5) { };
      \end{pgfonlayer}
      \begin{pgfonlayer}{edgelayer}
	\draw [arrow] (X) to (a1);
	\draw [arrow] (Y) to (b1);
	\draw [arrow] (Z) to (b2);
	\draw [arrow, out = 90, in = 270] (m1) to (a2);
	\draw [arrow, out = 90, in = 270] (m2) to (out);
      \end{pgfonlayer}
    \end{tikzpicture}
  \end{aligned}\,,
  \qquad\qquad
  \begin{aligned}
    \begin{tikzpicture}
      \begin{pgfonlayer}{nodelayer}
	\node [style=box] (m) at (0,0) {\;$m$\;};
	\node (a) at (-0.2,0) {};
	\node (b) at (0.2,0) {};
	\node [style = box] (X) at (-0.2,-0.8) {$e$ };
	\node (Y) at (0.2,-1.5) { };
	\node (out) at (0,1.5) { };
      \end{pgfonlayer}
      \begin{pgfonlayer}{edgelayer}
	\draw [arrow] (X) to (a);
	\draw [arrow] (Y) to (b);
	\draw [arrow, out = 90, in = 270] (m) to (out);
      \end{pgfonlayer}
    \end{tikzpicture}
  \end{aligned}
  \; = \;
  \begin{aligned}
    \begin{tikzpicture}
      \begin{pgfonlayer}{nodelayer}
	\node (out) at (0, 1.5) { };
	\node (in) at (0, -1.5) { };
      \end{pgfonlayer}
      \begin{pgfonlayer}{edgelayer}
	\draw [arrow] (in) to (out);
      \end{pgfonlayer}
    \end{tikzpicture}
  \end{aligned}
  \;= \;
  \begin{aligned}
    \begin{tikzpicture}
      \begin{pgfonlayer}{nodelayer}
	\node [style=box] (m) at (0,0) {\;$m$\;};
	\node (a) at (-0.2,0) {};
	\node (b) at (0.2,0) {};
	\node (X) at (-0.2,-1.5) { };
	\node [style = box] (Y) at (0.2,-0.8) {$e$};
	\node (out) at (0,1.5) { };
      \end{pgfonlayer}
      \begin{pgfonlayer}{edgelayer}
	\draw [arrow] (X) to (a);
	\draw [arrow] (Y) to (b);
	\draw [arrow, out = 90, in = 270] (m) to (out);
      \end{pgfonlayer}
    \end{tikzpicture}
  \end{aligned}\,,
\end{equation*}
expressing the associativity of $m$ and the fact that that $e$ is its identity. These conditions clearly define the class of all monoids when interpreted in $\cat{Rel}$, the allegory of sets and relations, and they define the class of all discrete quantum monoids when interpreted in $\cat{qRel}$ by \cite[Theorem~7.4]{zbMATH07287276}.

Classically, a group can be defined as a monoid that has all inverses or as a monoid that is equipped with an inversion map. The latter definition is common in mathematical logic \cite[\S 1.1, Example 3]{zbMATH00053151}, making a group into an algebraic structure, and in category theory \cite[\S III.6]{mcl_2e}, providing a definition internal to the category $\cat{Set}$. The former definition can be written diagrammatically as

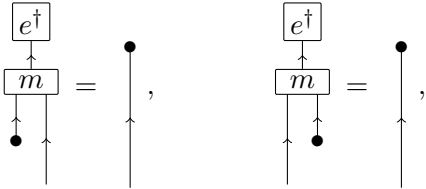
\vspace{-4.7ex}

\begin{equation*}
  \begin{aligned}
    \begin{tikzpicture}
      \begin{pgfonlayer}{nodelayer}
	\node [style=box] (m) at (0,0) {\;$m$\;};
	\node (a) at (-0.2,0) {};
	\node (b) at (0.2,0) {};
	\node [style = none] (X) at (-0.2,-0.8) {$\bullet$ };
	\node (Y) at (0.2,-1.5) { };
	\node [style = box] (end) at (0,0.8) {$e^\dagger$};
	\node (cloud) at (0,1.5) {};
      \end{pgfonlayer}
      \begin{pgfonlayer}{edgelayer}
	\draw [arrow] (X.center) to (a);
	\draw [arrow] (Y) to (b);
	\draw [arrow, out = 90, in = 270] (m) to (end);
      \end{pgfonlayer}
    \end{tikzpicture}
  \end{aligned}
  \; = \;
  \begin{aligned}
    \begin{tikzpicture}
      \begin{pgfonlayer}{nodelayer}
	\node (cloud) at (0, 1.5) {$\phantom{\bullet}$};
	\node (in) at (0, -1.5) { };
	\node [style = none] (end) at (0, 0.5) {$\bullet$};
      \end{pgfonlayer}
      \begin{pgfonlayer}{edgelayer}
	\draw [arrow] (in) to (end.center);
      \end{pgfonlayer}
    \end{tikzpicture}
  \end{aligned},
  \qquad\qquad
  \begin{aligned}
    \begin{tikzpicture}
      \begin{pgfonlayer}{nodelayer}
	\node [style=box] (m) at (0,0) {\;$m$\;};
	\node (a) at (-0.2,0) {};
	\node (b) at (0.2,0) {};
	\node (X) at (-0.2,-1.5) { };
	\node [style = none] (Y) at (0.2,-0.8) {$\bullet$};
	\node [style = box] (end) at (0,0.8) {$e^\dagger$};
	\node (cloud) at (0,1.5) {};
      \end{pgfonlayer}
      \begin{pgfonlayer}{edgelayer}
	\draw [arrow] (X) to (a);
	\draw [arrow] (Y.center) to (b);
	\draw [arrow, out = 90, in = 270] (m) to (end);
      \end{pgfonlayer}
    \end{tikzpicture}
  \end{aligned}
  \; = \;
  \begin{aligned}
    \begin{tikzpicture}
      \begin{pgfonlayer}{nodelayer}
	\node (cloud) at (0, 1.5) {$\phantom{\bullet}$};
	\node (in) at (0, -1.5) { };
	\node [style = none] (end) at (0, 0.5) {$\bullet$};
      \end{pgfonlayer}
      \begin{pgfonlayer}{edgelayer}
	\draw [arrow] (in) to (end.center);
      \end{pgfonlayer}
    \end{tikzpicture}
  \end{aligned},
\end{equation*}
where the ``loose end'' graphical element denotes the maximum morphism of that type. These two conditions clearly define the class of all groups when they are interpreted in $\cat{Rel}$, and they define the class of all discrete quantum groups when they are interpreted in $\cat{qRel}$ because they are equivalent to the two conditions in \Cref{thm:Vaes} by \cite[Theorem~B.8]{zbMATH07287276}.

The diagrammatic incarnation of the latter definition of a group, as a monoid that is equipped with an inversion map, is
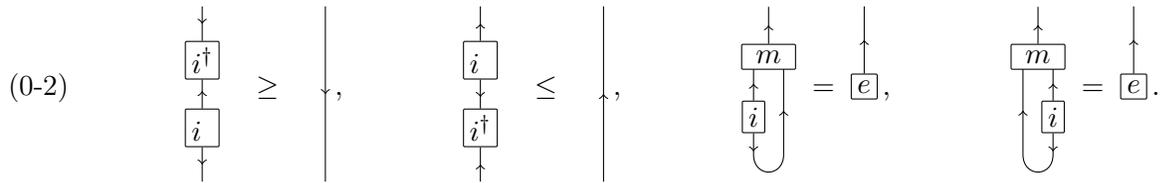
\begin{equation}\label{eq:3conds}
  \begin{aligned}
    \begin{tikzpicture}
      \begin{pgfonlayer}{nodelayer}
        \node (cloud) at (0,1.6) {};
	\node (in) at (0,-1.3) {};
	\node (out) at (0,1.3) {};
	\node [style=box](r) at (0, -0.45) {$i^{\phantom{\dagger}}$};
	\node [style=box](s) at (0, 0.45) {$i^\dagger$};
      \end{pgfonlayer}
      \begin{pgfonlayer}{edgelayer}
	\draw[arrow] (r) to (in);
	\draw[arrow, markat=0.6] (r) to (s);
	\draw[arrow] (out) to (s);
      \end{pgfonlayer}
    \end{tikzpicture}
  \end{aligned}
  \quad
  \geq
  \quad
  \begin{aligned}
    \begin{tikzpicture}
      \begin{pgfonlayer}{nodelayer}
        \node (cloud) at (0,1.6) {};
	\node (in) at (0,-1.3) {};
	\node (out) at (0,1.3) {};
      \end{pgfonlayer}
      \begin{pgfonlayer}{edgelayer}
	\draw[arrow] (out) to (in);
      \end{pgfonlayer}
    \end{tikzpicture}
  \end{aligned},
  \qquad \qquad
  \begin{aligned}
    \begin{tikzpicture}
      \begin{pgfonlayer}{nodelayer}
        \node (cloud) at (0,1.6) {};
	\node (in) at (0,-1.3) {};
	\node (out) at (0,1.3) {};
	\node [style=box](r) at (0, -0.45) {$i^\dagger$};
	\node [style=box](s) at (0, 0.45) {$i^{\phantom{\dagger}}$};
      \end{pgfonlayer}
      \begin{pgfonlayer}{edgelayer}
	\draw[arrow] (in) to (r);
	\draw[arrow, markat=0.6] (s) to (r);
	\draw[arrow] (s) to (out);
      \end{pgfonlayer}
    \end{tikzpicture}
  \end{aligned}
  \quad
  \leq
  \quad
  \begin{aligned}
    \begin{tikzpicture}
      \begin{pgfonlayer}{nodelayer}
        \node (cloud) at (0,1.6) {};
	\node (in) at (0,-1.3) {};
	\node (out) at (0,1.3) {};
      \end{pgfonlayer}
      \begin{pgfonlayer}{edgelayer}
	\draw[arrow] (in) to (out);
      \end{pgfonlayer}
    \end{tikzpicture}
  \end{aligned},
  \qquad \qquad
  \begin{aligned}
    \begin{tikzpicture}
      \begin{pgfonlayer}{nodelayer}
	\node [style=box] (m) at (0,0) {\;$m$\;};
	\node (a) at (-0.2,0) {};
	\node (b) at (0.2,0) {};
	\node [style = box] (X) at (-0.2,-0.8) {$i$ };
	\node (Y) at (0.2,-1.3) { };
        \node (Z) at (-0.2,-1.3) { };
	\node (end) at (0,0.8) {};
      \end{pgfonlayer}
      \begin{pgfonlayer}{edgelayer}
	\draw [arrow] (X) to (a);
	\draw [arrow] (Y.center) to (b);
	\draw [arrow] (m) to (end);
        \draw [arrow, markat=0.85] (X) to (Z.center);
        \draw [out = 270, in = 270, looseness = 2] (Z.center) to (Y.center);
      \end{pgfonlayer}
    \end{tikzpicture}
  \end{aligned}
  \;=\;
  \begin{aligned}
    \begin{tikzpicture}
      \begin{pgfonlayer}{nodelayer}
        \node (end) at (0,1.2) {};
	\node [style = box] (e) at (0,0){$e$};
        \node (cloud) at (0,-1.1) {};
      \end{pgfonlayer}
      \begin{pgfonlayer}{edgelayer}
        \draw[arrow] (e) to (end);
      \end{pgfonlayer}
    \end{tikzpicture}
  \end{aligned}\,,
  \qquad \qquad
  \begin{aligned}
    \begin{tikzpicture}
      \begin{pgfonlayer}{nodelayer}
	\node [style=box] (m) at (0,0) {\;$m$\;};
	\node (a) at (-0.2,0) {};
	\node (b) at (0.2,0) {};
	\node (X) at (-0.2,-1.3) { };
        \node (Z) at (0.2,-1.3) { };
	\node [style = box] (Y) at (0.2,-0.8) {$i$};
	\node (end) at (0,0.8) {};
      \end{pgfonlayer}
      \begin{pgfonlayer}{edgelayer}
	\draw [arrow] (X.center) to (a);
	\draw [arrow] (Y) to (b);
	\draw [arrow] (m) to (end);
        \draw [arrow, markat=0.85] (Y) to (Z.center);
        \draw [out = 270, in = 270, looseness = 2] (Z.center) to (X.center);
      \end{pgfonlayer}
    \end{tikzpicture}
  \end{aligned}
  \;=\;
  \begin{aligned}
    \begin{tikzpicture}
      \begin{pgfonlayer}{nodelayer}
        \node (end) at (0,1.2) {};
	\node [style = box] (e) at (0,0){$e$};
        \node (cloud) at (0,-1.1) {};
      \end{pgfonlayer}
      \begin{pgfonlayer}{edgelayer}
        \draw[arrow] (e) to (end);
      \end{pgfonlayer}
    \end{tikzpicture}
  \end{aligned}\,.
\end{equation}
The first pair of conditions expresses that $i$ is a map, and the second pair of conditions expresses that it is an inversion for $m$. Interpreted in $\cat{Rel}$ the four conditions define the class of all groups, and they recall the definition of a group object \cite[\S III.6]{mcl_2e}. That they define the class of all Kac-type quantum groups when interpreted in $\cat{qRel}$ is precisely the content of \Cref{thm:us}, as we verify in \Cref{le:trnsl.scond.bis}. This characterization is the conjecture that motivated \Cref{thm:us}.

\subsection*{Acknowledgments}

This work is part of the project Graph Algebras partially supported by EU grant HORIZON-MSCA-SE-2021 Project 101086394. This work was supported by the National Science Foundation under Award No.\ DMS-2231414.


\section{Diagonal states, antipodes, and discrete Kac quantum groups}\label{se:main}

For hereditarily atomic von Neumann algebras $M$ the tensor products $M\stensor N$ are unambiguous; minimal coincide with maximal by \cite[Proposition 8.6]{zbMATH03248454}. The functor $M\stensor \bullet$ is furthermore product-preserving \cite[Proposition 8.4]{zbMATH03248454}, so given a decomposition of $M$ of the form \Cref{eq:prod.dec}, we have 
\begin{equation}\label{eq:mmop}
  M\stensor M^{op}
  \cong
  \left(\prod_i^{W^*} M_{n_i}\right)\stensor \left(\prod_i^{W^*} M^{op}_{n_i}\right)
  \cong
  \prod_{i,j\in I}M_{n_i}\otimes M_{n_j}^{op}.
\end{equation}

We need some further preparation and reminders of various notions related to \emph{quantum sets} \cite[\S 1.3]{zbMATH07828321}, which are sets $\cX=\left\{X_{\alpha}\right\}_{\alpha}$ of finite-dimensional Hilbert spaces $X_{\alpha}$, referred to as the \emph{atoms} \cite[Definition 2.1]{zbMATH07287276} of $\cX$. \emph{Relations} between quantum sets appear in three guises, and we will take the one-to-one-to-one correspondence between these three guises for granted throughout:

\begin{enumerate}[(a),wide]
\item\label{item:qrel.k} on the one hand \cite{zbMATH07287276}, (binary) relations from a quantum set $\cX$ to a quantum set $\cY$, i.e., choices of subspaces
  \begin{equation*}
    R(X,Y)\le L(X,Y)
    :=
    \left\{\text{bounded operators }X\to Y\right\}
    ,\quad
    X\in \At(\cX),\quad Y\in \At(\cY)
  \end{equation*}
  for every $X$ and $Y$ (finite-dimensional Hilbert spaces) ranging over the sets $\At(\bullet)$ of atoms of $\cX$ and $\cY$, respectively, 
   
\item\label{item:qrel.w.bimod} on the other hand \cite{zbMATH07388954}, quantum relations  from $\ell^\infty(\cX)$ to $\ell^\infty(\cY)$, i.e., weak$^*$-closed subspaces
  \begin{equation}\label{eq:vlxly}
    \cV
    \le
    L\left(H_{\cX},\ H_{\cY}\right)
    ,\quad
    H_{\bullet}:=\bigoplus_{H\in \At(\bullet)}H
  \end{equation}
  with $\ell^{\infty}(\cY)'\cdot \cV\cdot \ell^{\infty}(\cX)'=\cV$, 

\item\label{item:qrel.w.proj} finally, projections in $\ell^{\infty}(\cX)\stensor \ell^{\infty}(\cY)^{op}$.
\end{enumerate}
The correspondence \Cref{item:qrel.k} $\leftrightsquigarrow$ \Cref{item:qrel.w.bimod} is worked out in \cite[\S A.2]{zbMATH07828321} and can be summarized as
\begin{equation*}
  \cV
  \text{ as in \Cref{item:qrel.w.bimod}}
  \xmapsto{\quad}
  \bigg(
  R(X,Y):=\id_{Y}\cdot \cV\cdot \id_{X}
  \bigg)_{\substack{X\in \At(\cX)\\Y\in \At(\cY)}}.
\end{equation*}
On the other hand, \Cref{item:qrel.w.bimod} $\leftrightsquigarrow$ \Cref{item:qrel.w.proj} is essentially \cite[Proposition 2.23]{zbMATH06008057}, with appropriate routine modifications (that proposition handles \emph{self}-relations on a single \emph{finite} quantum set). It is more convenient to describe \Cref{item:qrel.k} $\leftrightsquigarrow$ \Cref{item:qrel.w.proj} instead:
\begin{align*}
  \left(R(X,Y)\right)_{X,Y}
  &\xmapsto{\quad}
    \text{projection $P$ with }\Ann_{A}\left(
    \cR:=\prod_{X,Y} R(X,Y)
    \right)
    =
    (1-P)A,\numberthis\label{eq:rel2proj.inf}\\
  \Ann_{A}(\cR)
  &:=
    \left\{a\in A\ :\ \cR a=\{0\}\right\},
\end{align*}
where
\begin{equation*}
    A
    :=
    \ell^{\infty}(\cX)\stensor \ell^{\infty}(\cY)^{op}
    \cong
    \prod_{\substack{X\in \At(\cX)\\Y\in \At(\cY)}}^{W^*}
    \End(X)\otimes \End(Y)^{op}
  \end{equation*}
right-acts on the full algebraic product
\begin{equation*}
  \prod_{\substack{X\in \At(\cX)\\Y\in \At(\cY)}}
  L(X,Y)
\end{equation*}
by composition
  \begin{equation*}
    x(a\otimes b)
    :=
    bxa
    ,\quad
    a\otimes b\in
    \End(X)\otimes \End(Y)^{op}
    ,\quad
    x\in \Hom(X,Y),
  \end{equation*}
with all other factors annihilating each other. In framework \Cref{item:qrel.w.bimod}, the composition of quantum relations $\cV$ and $\cW$ may be defined as the weak$^*$-closed span of the products $vw$ for $v \in \cV$ and 
$w \in \cW$, and the dagger adjoint of $\cV$ may be defined as the subspace of adjoints $v^*
$ for $v \in \cV$. See \cite[\S~3]{zbMATH07287276} for the corresponding definitions in framework \Cref{item:qrel.k}.

\begin{remark}\label{re:rel2proj.inf}
  That \Cref{eq:rel2proj.inf} implements a bijection between \Cref{item:qrel.k} and \Cref{item:qrel.w.proj} above for arbitrary quantum sets is not difficult to deduce from the aforementioned \cite[Proposition 2.23]{zbMATH06008057}.

  For a quantum set $\cX=\left\{X_{\alpha}\right\}_{\alpha}$ write $\cX_{F}$ for the quantum set comprising a subset $F\subseteq \At(\cX)$ of atoms. There are obvious relations $\cX\leftrightarrow \cX_{F}$ realizing $\cX$ as both the categorical product and the \emph{co}product \cite[Remark 3.7]{zbMATH07287276} of the atomic quantum sets $\cX_{\alpha}$, each consisting of a single atom $X_{\alpha}$. Now, for quantum sets $\cX$ and $\cY$ we have
  \begin{equation*}
    \begin{tikzpicture}[>=stealth,auto,baseline=(current  bounding  box.center)]
      \path[anchor=base] 
      (0,0) node (l) {$\cat{qRel}(\cX,\cY)$}
      +(-2,2) node (ul) {$\displaystyle\varprojlim_{\substack{\text{finite $F\subseteq \At(\cX)$}\\\text{finite $G\subseteq \At(\cY)$}}}\cat{qRel}(\cX_F,\cY_G)$}
      +(8,2) node (ur) {$\displaystyle\varprojlim_{F,G}
    \mathrm{Proj}\left(\ell^{\infty}(\cX_F)\stensor \ell^{\infty}(\cY_G)^{op}\right)$}
      +(6,0) node (r) {$\mathrm{Proj}\left(\ell^{\infty}(\cX)\stensor \ell^{\infty}(\cY)^{op}\right)$}
      ;
      \draw[->] (ul) to[bend left=6] node[pos=.5,auto,swap] {$\scriptstyle \cong$} node[pos=.5,auto] {$\scriptstyle \text{\cite[Prop. 2.23]{zbMATH06008057}+\cite[\S A.2]{zbMATH07828321}}$} (ur);
      \draw[->] (l) to[bend left=6] node[pos=.5,auto] {$\scriptstyle \cong$} (ul);
      \draw[->] (ur) to[bend left=6] node[pos=.5,auto] {$\scriptstyle \cong$} (r);
    \end{tikzpicture}
  \end{equation*}
  with the left-hand isomorphism being a consequence of $\cX$ being the (co)product of its atomic components.
\end{remark}

Some notation will help handle the passage between the various avatars of a quantum relation.  

\begin{notation}\label{not:rel.proj.fn}
  We denote
  \begin{enumerate}[(1), wide]
  \item the relation associated with a $W^*$-morphism by
    \begin{equation}\label{eq:mor2rel}
      \left(\ell^{\infty}(\cY)\xrightarrow{\quad \psi \quad}\ell^{\infty}(\cX)\right)
      \quad
      \xrsquigarrow{\qquad}
      \quad
      \left(\cX\xrightarrow{\quad \hat \psi\quad} \cY\right);
    \end{equation}

  \item the relation associated to a projection by 
    \begin{equation*}
      \mathrm{Proj}\left(
        \ell^{\infty}(\cX)\stensor \ell^{\infty}(\cY)^{op}
      \right)
      \ni
      p
      \quad
      \xrsquigarrow{\qquad}
      \quad
      \left( \cX\xrightarrow{\quad \hat p\quad} \cY \right);
    \end{equation*}

  \item the projection associated to a relation by
    \begin{equation*}
      \left(\cX\xrightarrow{\quad R\quad} \cY\right)
      \quad
      \xrsquigarrow{\qquad}
      \quad
      \left(
        P_R
        \in
        \ell^{\infty}(\cX)\stensor \ell^{\infty}(\cY)^{op}
      \right);
    \end{equation*}

  \item and finally (following \cite[\S 2.2]{zbMATH07828321}), by $\cY\xrightarrow{\top_{\cY}^{\cZ}}\cZ$ the largest relation between two quantum sets, with the monoidal unit $\mathbf{1}$ omitted as either a subscript or superscript.
  \end{enumerate}
\end{notation}

\begin{remark}\label{re:intertw}
  We remind the reader (e.g. \cite[Theorem]{1101.1694v3} or \cite[Theorem 6.3]{zbMATH07287276}) that the correspondence \Cref{eq:mor2rel} is given by
  \begin{equation*}
    \hat \psi
    =
    \cV:=\left\{T\in L\left(H_{\cX},H_{\cY}\right)\ :\ aT=T\psi(a),\ \forall a\in \ell^{\infty}(\cY)\right\}
  \end{equation*}
  in the notation of \Cref{eq:vlxly}. In other words, $\cV$ is the space of all bounded operators that \emph{intertwine} two actions of $\ell^{\infty}(\cY)$.
\end{remark}

The \Cref{item:qrel.k} + \Cref{item:qrel.w.bimod} + \Cref{item:qrel.w.proj} picture can be further amplified by identifying \cite[Theorem 7.4]{zbMATH07287276} $W^*$-morphisms $\ell^{\infty}(\cY)\to \ell^{\infty}(\cX)$ with \emph{functions} or synonymously \emph{maps} $\cX\to \cY$ \cite[Definition 4.5]{zbMATH07287276}, which are relations  $\cX\xrightarrow{R}\cY$ such that
\begin{equation*}
  \id_{\cY}\ge R\circ R^{\dag}
  \quad\text{and}\quad
  \id_{\cX}\le R^{\dag}\circ R. 
\end{equation*}

Recall also \cite[Theorem 3.6]{zbMATH07287276} that the monoidal product on the category $\cat{qRel}$ of quantum sets and relations makes the latter \emph{rigid} \cite[Definition 2.10.1]{egno} symmetric monoidal and therefore \emph{monoidal closed} in the sense of \cite[Definition 6.1.3]{brcx_hndbk-2}. This monoidal product is called the Cartesian product and is written $\cX \times \cY$ because it generalizes the Cartesian product of sets, but it is not the category-theoretic product. We denote by
\begin{equation*}
  \cZ^*\times \cZ
  \xrightarrow{\quad\ev_{\cZ}\quad}
  \mathbf{1}
  :=
  \text{monoidal unit of $\cat{qRel}$}
\end{equation*}
the evaluation morphisms underlying the rigid structure of $\cat{qRel}$. 

The homomorphism-function correspondence interacts well with the projection incarnation of relations between quantum sets:

\begin{proposition}\label{pr:proj.compat.mor}
  Given 
  \begin{equation*}
    \cX\xrightarrow[\quad\text{relation}\quad]{\quad R\quad} \cY
    \quad\text{and}\quad
    \begin{aligned}
      \ell^{\infty}(\cY)
      &\xrightarrow[\quad\text{$W^*$-morphism}\quad]{\quad\psi_{r}\quad}
        \ell^{\infty}(\cZ)\\
      \ell^{\infty}(\cX)
      &\xrightarrow[\quad\text{$W^*$-morphism}\quad]{\quad\psi_{\ell}\quad}
        \ell^{\infty}(\cW)\\
    \end{aligned}
  \end{equation*}
  we have
  \begin{equation}\label{eq:mor.compat.proj}
    \begin{tikzpicture}[>=stealth,auto,baseline=(current  bounding  box.center)]
      \path[anchor=base] 
      (0,0) node (lu) {$\ell^{\infty}(\cX)\stensor \ell^{\infty}(\cY)^{op}$}
      +(5,0) node (ru) {$\ell^{\infty}(\cW)\stensor \ell^{\infty}(\cZ)^{op}$}
      +(0,-1.5) node (ld) {$P_R$}
      +(5,-1.5) node (rd) {$P_{\hat \psi_{r}^{\dag}\circ R\circ \hat \psi_{\ell}}$}    
      +(0,-.7) node () {$\text{\rotatebox[origin=c]{90}{$\in$}}$}
      +(5,-.7) node () {$\text{\rotatebox[origin=c]{90}{$\in$}}$}
      ;
      \draw[->] (lu) to[bend left=6] node[pos=.5,auto] {$\scriptstyle \psi_{\ell}\stensor \psi_{r}$} (ru);
      \draw[|->] (ld) to[bend right=6] node[pos=.5,auto] {$\scriptstyle $} (rd);
    \end{tikzpicture}
  \end{equation}
  in the language of \Cref{not:rel.proj.fn}.
\end{proposition}

\Cref{pr:proj.compat.mor} follows from basic results about quantum sets, namely \cite[Theorems 3.6, 7.4, and B.8]{zbMATH07287276}. Explicitly, we have natural bijections
\begin{equation*}
\cat{qRel}(\cX, \cY)
\iso
\cat{qRel}(\cX \times \cY^*, \mathbf 1)
\iso
 \mathrm{Proj}(\ell^\infty(\cX \times \cY^*))
 \iso
 \mathrm{Proj}(\ell^\infty(\cX) \stensor \ell^\infty(\cY)^{op}),
\end{equation*}
where the first natural bijection comes from \cite[Theorem~3.6]{zbMATH07287276}, the second from \cite[Theorem~B.8]{zbMATH07287276}, and the third from \cite[Theorem~7.4]{zbMATH07287276}. We provide an alternative proof of \Cref{pr:proj.compat.mor} that is closer to the basic notions.

\pf{pr:proj.compat.mor}
\begin{pr:proj.compat.mor}
  A number of simplifications will boil down the claim to its core. First, it will suffice to handle each of $\psi_{\ell}$ and $\psi_r$ separately, assuming that the other is the identity. Second, the two resulting branches are entirely analogous (and in fact interchangeable by applying the dagger functor of \cite[Definition 3.5]{zbMATH07287276}). It thus suffices to work with three quantum sets $\cW$, $\cX$, and $\cY=\cZ$ and the single morphism $\psi:=\psi_{\ell}$. Finally, the natural bijections
  \begin{equation*}
    \cat{qRel}\left(\bullet,\cY\right)
    \quad\cong\quad
    \cat{qRel}\left(\bullet\times \cY^*,\mathbf{1}\right)
  \end{equation*}
  witnessing rigidity further reduce the problem to the case $\cY=\textbf{1}$, after relabeling
  \begin{equation*}
    \cX\times \cY^*\xrsquigarrow{\quad} \cX
    \quad\text{and}\quad
    \cW\times \cY^*\xrsquigarrow{\quad} \cW.
  \end{equation*}
  In light of \Cref{re:intertw}, the claim now amounts to \Cref{le:proj.mor} below.
\end{pr:proj.compat.mor}

\begin{lemma}\label{le:proj.mor}
  Let $N\le L(K)$ and $M\le L(H)$ be two von Neumann algebras, $N\xrightarrow{\psi}M$ a $W^*$-morphism, and $p\in N$ a projection. Also let
  \begin{equation*}
    \cV:=\left\{T\in L(K,H)\ :\ Tn=\psi(n)T,\ \forall n\in N\right\}. 
  \end{equation*}
  The common support projection of the subspaces $T\im p \leq H$, for $T\in \cV$, is precisely $\psi(p)$. 
\end{lemma}
\begin{proof}
  That support projection is plainly dominated by $\psi(p)$, so the claim is that it cannot be smaller. Passing to the morphism $pNp\to \psi(p)M\psi(p)$ induced by $\psi$, there is no loss in assuming $p=1$. The conclusion now follows from \cite[Proposition 1.4]{1101.1694v3}, which (translated into our notation) shows that $\cV\cV^*$ (weak$^*$-closed span of products) contains $1$.
\end{proof}

We can now begin to unwind some of the formalism underlying \Cref{thm:us,th:is.kac.bis}. 

\begin{lemma}\label{le:trnsl.scond}
  Given $(M=\ell^{\infty}(\cX),\Delta,\varepsilon)$ as in \Cref{thm:us}, the conditions \Cref{item:thm:us:l} and \Cref{item:thm:us:r} imposed on $s$ are respectively equivalent to
  \begin{enumerate}[(a),wide]
  \item\label{item:le:trnsl.scond:l} $\hat \varepsilon^{\dag} \circ \hat \Delta \circ (\hat s\times\id) \ge \ev_{\cX}$,
  \item\label{item:le:trnsl.scond:r} $ \hat \varepsilon^{\dag} \circ \hat \Delta \circ (\id\times \hat s) \ge \ev_{\cX^*}$.
  \end{enumerate}
\end{lemma}
\begin{proof}
  We handle the left-handed version concerning \Cref{item:le:trnsl.scond:l}, the other claim being entirely parallel.

  Denote by $e\in M$ the support projection of $\varepsilon$. We have
  \begin{equation}\label{eq:long.eq}
    \begin{aligned}
      \text{condition \Cref{item:thm:us:l} of \Cref{thm:us}}
      &\xLeftrightarrow{\quad}
        \left(\varphi\circ (s\stensor\id)\circ\Delta\right)(1-e)=0
        ,\quad\forall\text{ diagonal state }\varphi\\
      &\xLeftrightarrow{\quad}
        \varphi\left(\left((s\stensor\id)\circ\Delta\right)(1-e)\right)=0
        ,\quad\forall\text{ diagonal state }\varphi\\
      &\xLeftrightarrow{\quad}
        \left((s\stensor\id)\circ\Delta\right)(1-e) \perp \delta_{M^{op}}\\
      &\xLeftrightarrow{\quad}
        \left((s\stensor\id)\circ\Delta\right)(e) \ge \delta_{M^{op}}.\\
    \end{aligned}
  \end{equation}
  
  The right-hand side's $\delta_{M^{op}}$ is nothing but $P_{\ev_{\cX}}$, while
  \begin{equation*}
    e=P_{\hat \varepsilon^{\dag}}
    \xRightarrow{\quad\text{\Cref{eq:mor.compat.proj}}\quad}
    \left((s\stensor\id)\circ\Delta\right)(e) = P_{\hat \varepsilon^{\dag} \circ \hat \Delta \circ (\hat s\times\id)}.
  \end{equation*}
  The latter dominating $\delta_{M^{op}}$ means precisely \Cref{item:le:trnsl.scond:l} because $\bullet\mapsto P_{\bullet}$ is order-preserving, finishing the proof. 
\end{proof}

\begin{remark}\label{re:cj}
  The inequalities \Cref{item:le:trnsl.scond:l} and \Cref{item:le:trnsl.scond:r} of \Cref{le:trnsl.scond} are what the two bottom constraints in \cite[Conjecture, p.17]{kor_qgr_sld}, formulated there in graphical language, unpack to. In other words, \Cref{le:trnsl.scond} bridges the gap between that conjectural statement and that of \Cref{thm:us}, which thus resolves the conjecture. 
\end{remark}

The following variant of \Cref{le:trnsl.scond} will also confirm that, as we remarked in the Introduction, the constraints imposed in \Cref{thm:us} amount to precisely the interesting portion of \Cref{eq:3conds}.

\begin{lemma}\label{le:trnsl.scond.bis}
  Given $(M=\ell^{\infty}(\cX),\Delta,\varepsilon)$ as in \Cref{thm:us}, the conditions \Cref{item:thm:us:l} and \Cref{item:thm:us:r} imposed on $s$ are respectively equivalent to
  \begin{enumerate}[(a),wide]
  \item\label{item:le:trnsl.scond.bis:l} $\hat \Delta \circ (\hat s\times\id)\circ \ev^{\dag}_{\cX}
    =
    \hat \varepsilon$,
  \item\label{item:le:trnsl.scond.bis:r} $\hat \Delta \circ (\id\times \hat s)\circ \ev^{\dag}_{\cX^*}
    =
    \hat \varepsilon$,
  \end{enumerate}
  i.e., the third and fourth diagrams in \Cref{eq:3conds}. 
\end{lemma}
\begin{proof}
  Proving that the diagrams indeed amount precisely to \Cref{item:le:trnsl.scond.bis:l,item:le:trnsl.scond.bis:r} as indicated is a simple matter of recalling the former's very definition. Given \Cref{le:trnsl.scond}, it remains to equate \emph{its} conditions \Cref{item:le:trnsl.scond:l,item:le:trnsl.scond:r} to the present ones. The two being entirely parallel, we argue only the \Cref{item:le:trnsl.scond.bis:l} branch of the claim. To that end, simply note that for
  \begin{equation*}
    R
    :=
    \hat \Delta \circ (\hat s\times\id),
  \end{equation*}
  we have
  \begingroup
  \allowdisplaybreaks
  \begin{align*}
   \hat \varepsilon^{\dag}\circ R
    \ge
    \ev_{\cX}
    &\xLeftrightarrow{\quad\text{$R$ is a map}\quad}
     \hat \varepsilon^{\dag}
      \ge
      \ev_{\cX}\circ R^{\dag}\\
    &\xLeftrightarrow{\quad\text{$\hat \varepsilon$ is minimal}\quad}
     \hat \varepsilon^{\dag}
      =
      \ev_{\cX}\circ R^{\dag}\\
    & \Longleftrightarrow \text{condition \Cref{item:le:trnsl.scond.bis:l}$^{\dag}$}
  \end{align*}
  \endgroup
  and hence the conclusion. 
\end{proof}

Note that the argument proving \Cref{le:trnsl.scond} also effectively reduces \Cref{thm:us} to \Cref{th:is.kac.bis} (as announced in the Introduction), via \Cref{thm:Vaes}:

\begin{lemma}\label{le:supp.proj.dom}
  Given $(M=\ell^{\infty}(\cX),\Delta,\varepsilon)$ as in \Cref{thm:us}, the conditions \Cref{item:thm:us:l} and \Cref{item:thm:us:r} imposed on $s$ imply \Cref{item:thm:Vaes:mur} and \Cref{item:thm:Vaes:mul} of \Cref{thm:Vaes}, respectively.
  In particular, a quantum monoid satisfying the hypotheses of \Cref{thm:us} is a discrete quantum group. 
\end{lemma}
\begin{proof}
  Observe first that \Cref{thm:Vaes}\Cref{item:thm:Vaes:mur} is equivalent to
  \begin{equation*}
    \left(1\stensor p_{\nu}\right)\Delta(e)\ne 0
    \quad\text{with}\quad
    p_{\nu}:=\text{support projection of }\nu,
  \end{equation*}
  given that the (positive \cite[Theorem 2]{tom-prod}) element $\left(\id\stensor\nu\right)\left(\Delta(e)\right)$ will be annihilated by all normal states precisely when it vanishes. It will thus suffice to prove that
  \begin{equation*}
    \text{\Cref{thm:us}\Cref{item:thm:us:l}}
    \xRightarrow{\quad}
    \left(1\stensor p\right) \wt{e}
    =
    \left(s\stensor\id\right)
    \left(\left(1\stensor p\right)\Delta(e)\right)
    \ne 0
    ,\quad
    \forall\text{ projection }p\ne 0,
  \end{equation*}
  where $\wt{e}:=\left((s\stensor\id)\circ\Delta\right)(e)$. We can prove this implication by reasoning, for any projection $p$ such that $\left(1\stensor p\right) \wt{e} = 0$, that

  \begin{equation*}
    \begin{aligned}
      \text{\Cref{thm:us}\Cref{item:thm:us:l}}
      \xRightarrow{\quad \text{\Cref{eq:long.eq}}\quad}    
      \wt{e}\ge \delta_M
      &\xRightarrow{\quad \left(1\stensor p\right) \wt{e}=0\quad}
        \left(p\stensor p\right)\delta_M=0\\
      &\xRightarrow{\quad\text{\cite[Proposition A.1.2]{zbMATH07828321}}\quad}
      p=0,
    \end{aligned}    
  \end{equation*}
  thus verifying the claim.
\end{proof}

\pf{th:is.kac.bis}
\begin{th:is.kac.bis}
Let $\cX$ be a quantum set such that $M\cong \ell^{\infty}(\cX)$; the atoms of $\cX$ are finite-dimensional Hilbert spaces that correspond to the factors of $M$. Because $(M,\Delta,\varepsilon)$ is a discrete quantum group, the atoms $X_{\beta}\in \At(\cX)$ also correspond to the irreducible representations $\beta$ of the compact quantum \emph{dual group} \cite[\S 3]{zbMATH04152742}, with each atom being a carrier of the corresponding representation.
We now write
\begin{itemize}[wide]
\item $X_{\bf 1}$ for (the carrier space of) the trivial representation;

\item $X_{\beta^*}$ for the dual representation $X^*_{\beta}$.
\end{itemize}
Without loss of generality, $M = \ell^\infty(\cX) \leq L(H)$, where $H = \bigoplus_\alpha X_\alpha$, as in \cref{eq:vlxly}.

\Cref{le:trnsl.scond} provides the inequalities
\begin{equation}\label{eq:sld.diags}
 \hat \varepsilon^{\dag}
  \circ
  \hat \Delta
  \circ
  (\hat s\times\id)
  \ge
  \ev_{\cX}
  \quad\text{and}\quad
 \hat \varepsilon^{\dag}
  \circ
  \hat \Delta
  \circ
  (\id\times \hat s)
  \ge
  \ev_{\cX^*}.
\end{equation}
Since $\hat \varepsilon$ is a function, i.e., a morphism in $\cat{qRel}$ that is associated to a $W^*$-morphism, we also have
\begin{equation*}
  \id\ge \hat \varepsilon\circ\hat \varepsilon^{\dag}
  \quad\text{and}\quad
  \id\le\hat \varepsilon^{\dag}\circ \hat \varepsilon. 
\end{equation*}
These render \Cref{eq:sld.diags} equivalent to 
\begin{equation}\label{eq:sld.diags.movepsilon}
  \hat \Delta
  \circ
  (\hat s\times\id)
  \ge
  \hat \varepsilon
  \circ
  \ev_{\cX}
  \quad\text{and}\quad
  \hat \Delta
  \circ
  (\id\times \hat s)
  \ge
  \hat \varepsilon
  \circ
  \ev_{\cX^*},
\end{equation}
respectively, with the left-hand sides of the inequalities again maps.

The dagger adjoint of $\hat \varepsilon \circ \ev_{\cX^*}$ in \Cref{eq:sld.diags.movepsilon} is the relation
\begin{equation}\label{eq:rxaxb}
  R(X_{\alpha},\ X_{\beta}\otimes X_{\gamma})
  =\bC \delta_{\mathbf{1},\alpha}\delta_{\gamma,\beta^*}\db_{\beta},
\end{equation}
where
\begin{equation}\label{eq:db}
  \bC
  \ni 1
  \xmapsto{\quad\db_{\beta}\quad}
  \sum_{\text{basis $e_i\in X_{\beta}$}} e_i\otimes e_i^*
  \in
  X_{\beta}\otimes X^*_{\beta}
  =
  X_{\beta}\otimes X_{\beta^*}  
\end{equation}
is the usual coevaluation (the symbol stands for \emph{dual basis}). The dagger adjoint of $\hat \varepsilon \circ \ev_{\cX}$ may be described similarly.

Without loss of generality, we focus on the right-hand inequality in \Cref{eq:sld.diags.movepsilon} and denote by
\begin{equation*}
  \begin{tikzpicture}[>=stealth,auto,baseline=(current  bounding  box.center)]
    \path[anchor=base] 
    (0,0) node (l) {$M$}
    +(2,.5) node (u) {$M\stensor M$}
    +(4,0) node (r) {$M\stensor M^{op}$}
    ;
    \draw[->] (l) to[bend left=6] node[pos=.5,auto] {$\scriptstyle \Delta$} (u);
    \draw[->] (u) to[bend left=6] node[pos=.5,auto] {$\scriptstyle \id\stensor s$} (r);
    \draw[->] (l) to[bend right=6] node[pos=.5,auto,swap] {$\scriptstyle \psi$} (r);
  \end{tikzpicture}
\end{equation*}
the corresponding $W^*$-morphism.
By \Cref{re:intertw}, the quantum relation $\cV\leq L(H \tensor H^*,H)$ that is associated with the $W^*$-morphism $\psi$
is the space 
\begin{equation*}
  \cV:=\left\{T\in L(H \tensor H^*,H)\ :\ mT=T\psi(m),\ \forall m\in M\right\}
\end{equation*}
of intertwiners for the two $M$-representations on $H$ and $H \tensor H^*$. The right-hand inequality of \Cref{eq:sld.diags.movepsilon} requires precisely that the operators $\db_\beta$ of \Cref{eq:db} satisfy $\db_\beta^\dagger \in \cV$, i.e., that
\begin{equation*}
\psi(m) \db_\beta = \db_\beta m
\end{equation*}
for all $m \in M$ and all irreducible representations $\beta$ of the compact quantum dual group of $(M, \Delta, \varepsilon)$.

The operators $\db_\beta$ also satisfy $\db_\beta m = \varepsilon(m) \db_\beta$ for all $m \in M$ because they all satisfy $\db_\beta e = \db_\beta$. Thus, the right-hand inequality of \Cref{eq:sld.diags.movepsilon} is equivalent to
\begin{equation}\label{eq:scale.db}
  \left(\forall m\in M\right)
  \left(\forall \beta\right)
  \quad:\quad
  \psi(m)\db_{\beta}=\varepsilon(m)\db_{\beta}
  ,  
\end{equation}
where we may identify the operator $\db_\beta$ of \Cref{eq:db} with the vector $\db_\beta(1) = \sum e_i\otimes e_i^* \in X_\beta \tensor X_\beta^*$.

The action of
\begin{equation*}
  L(X_{\beta})\otimes L(X_{\beta})^{op}
  \le
  M\stensor M^{op}
\end{equation*}
on
\begin{equation}\label{eq:hbetahbeta}
  X_{\beta}\otimes X_{\beta}^*\cong L(X_{\beta})
\end{equation}
is nothing but multiplication: $(a\otimes b)\triangleright x:=axb $. The isomorphism \Cref{eq:hbetahbeta} identifies the element $\db_{\beta}\in X_{\beta}\otimes X_{\beta}^*$ with the identity $1 \in L(X_\beta)$, so that \Cref{eq:scale.db} simply says that
\begin{equation}\label{eq:char.antp}
  \forall m\in M
  \quad:\quad
  \mathrm{mult}\left(((\id \stensor s) \circ \Delta) (m)\right)
  =
  \varepsilon(m)1\in M,
\end{equation}
where $\mathrm{mult}\: M \stensor M^{op} \to \mathrm{Mult}(M)$ is the multiplication/composition map and $\mathrm{Mult}(M) \iso \ell(\cX)$ is the multiplier algebra of $M$; see \cite[Definition~5.1]{zbMATH07287276}. Therefore, the right-hand inequality in \Cref{eq:sld.diags.movepsilon} is equivalent to \Cref{eq:char.antp}, and we have a similar equivalence for the left-hand inequality.

Recalling the notation $M_{00}$ for the algebraic direct sum $\bigoplus_{\alpha}L(H_{\alpha})\le M$, $s$ restricts to a $*$-morphism $M_{00}\to M$, which satisfies \Cref{eq:char.antp} and its $s$-on-the-left variant iff $s$ is the antipode of $(M, \Delta, \varepsilon)$ (as noted, for instance, immediately following \cite[Theorem 4.6]{zbMATH00569708}). Therefore, if $s$ satisfies \Cref{eq:char.antp} and its left variant, then $s$ restricts to the antipode, which is then a $*$-map, and if the antipode is a $*$-map, then it extends to a $W^*$-morphism $s$ that satisfies \Cref{eq:char.antp} and its left variant.

We had already shown that \Cref{eq:char.antp} and its left variant are together equivalent to the pair of inequalities in \Cref{eq:sld.diags}, which are equivalent to conditions \Cref{item:th:is.kac.bis:l} and \Cref{item:th:is.kac.bis:r} of \Cref{th:is.kac.bis} by \Cref{le:trnsl.scond}, so the theorem is proved.
\end{th:is.kac.bis}

Given their importance in the above discussion, it will be of some interest to have a fairly hands-on description for the diagonal states on $M\stensor M^{op}$.

\begin{proposition}\label{pr:diag.st}
  For a hereditarily atomic von Neumann algebra $M$ with a decomposition \Cref{eq:prod.dec}, the diagonal states on $M \stensor M^{op}$ with the decomposition \Cref{eq:mmop} are precisely those of the form
  \begin{equation*}
    \sum_{i\in I}c_i \mathrm{tr}_i\circ\mathrm{mult}_i
    ,\quad
    c_i\ge 0
    ,\quad
    \sum_i c_i=1,
  \end{equation*}
  where $M_{n_i}\xrightarrow{\tr_i}\bC$ and $M_{n_i}\stensor M^{op}_{n_i}\xrightarrow{\mathrm{mult}_i}M_{n_i}$ are the normalized trace and the multiplication maps, respectively. 
\end{proposition}
\begin{proof}
  That diagonal states annihilate the cross-factors $M_{n_i}\otimes M^{op}_{n_j}$, $i\ne j$, of \Cref{eq:mmop} follows immediately from $\varphi(p_i\otimes (1-p_i))=0$ for the central projection $p_i$ that is the unit of the factor $M_{n_i}$. The normality of $\varphi$ implies \cite[Corollary III.3.11]{tak1} that there is a convex-combination decomposition
  \begin{equation*}
    \varphi
    =
    \sum_{i\in I}c_i \varphi_i
    ,\quad
    c_i\ge 0
    ,\quad
    \sum_i c_i=1
  \end{equation*}
  for states $\varphi_i$ on $M_{n_i}\otimes M_{n_i}^{op}$, so it is enough to prove that $\mathrm{tr}\circ\mathrm{mult}\: M_n\otimes M_n^{op} \to M_n$ is the unique diagonal state on $M_n\otimes M_n^{op}$. 

  The $M_n$-bimodule structure on $M_n$ realizes the latter when it is equipped with the inner product
  \begin{equation*}
    \braket{a\mid b}:=\tr(a^*b),
  \end{equation*}
  being the carrier Hilbert space of a \emph{standard representation} \cite[Definition IX.1.13]{tak2} of $M_n$  and hence also of a representation of $M_n\otimes M_n^{op}$. Our state $\varphi$ is a convex combination of vector states
  \begin{equation*}
    M_n\otimes M_n^{op}
    \ni
    x\otimes y
    \xmapsto{\quad\varphi_a\quad}
    \tr(a^* xay),
  \end{equation*}
  each of them diagonal if $\varphi$ is, because each will be dominated by a scalar multiple of $\varphi$. For $\varphi_a$, being diagonal reads
  \begin{equation*}
    \tr\left(a^* p a (1-p)\right)=\tr\left((pa(1-p))^*(pa(1-p))\right)=0
    \xRightarrow{\quad}
    pa(1-p)=0
    ,\quad
    \forall \text{ projection }p\in M_n. 
  \end{equation*}
  This means that $a^\dagger\in M_n$ leaves all the subspaces $\im p\le \bC^n$ invariant and so must be a scalar. 
\end{proof}

\begin{remark}\label{re:lideal.qrel}
  The classification of diagonal states in \Cref{pr:diag.st} connects well to the theory of quantum relations developed in \cite{zbMATH06008057}.

  To expand, consider a finite-dimensional von Neumann algebra $M$ and represent $M\otimes M^{op}$ on $L^2(M,\tau)$ for a faithful tracial state $\tau$. A diagonal state on $M\otimes M^{op}$ is a linear combination of vector states associated with vectors annihilated by the left ideal generated by the projections $p\otimes (1-p)$; that annihilation implies centrality in $M$.

  This left ideal, which is nothing but the kernel of the multiplication map $M\otimes M^{op}\to M$ \cite[Theorem 0.1]{MR4941556}, corresponds via \cite[Proposition 2.23]{zbMATH06008057} to the \emph{equality (quantum) relation} of \cite[\S 1.1, p.338]{zbMATH07828321}. A justification for the term and a manifestation of how that relation represents equality is evident in \Cref{le:distinguish.morphisms}, which restates \cite[Proposition~3.5.4]{zbMATH07828321} in elementary terms.
\end{remark}

\begin{lemma}\label{le:distinguish.morphisms}
  Two $W^*$-morphisms $M\xrightarrow{f,g}N$ between hereditarily atomic von Neumann algebras are equal if and only if the composition
  \begin{equation*}
    M\stensor M^{op}
    \xrightarrow{\quad f\stensor g\quad}
    N\stensor N^{op}
    \xrightarrow{\quad\varphi\quad}
    \bC
  \end{equation*}
  is diagonal for all diagonal states $\varphi$.   
\end{lemma}
\begin{proof}
  By the classification in \Cref{pr:diag.st}, there is no loss of generality in assuming $N\cong M_n$. Diagonal states of course restrict to diagonal states along embeddings, so we can also set $M\cong \bC^2$. Morphisms then simply select projections $p,q\in M_n$, and the claim is just that
  \begin{equation*}
    \tr(p(1-q)) = 0 = \tr((1-p)q)
    \xRightarrow{\quad}
    p=q.
  \end{equation*}
  Indeed, the hypothesis implies that $p(1-q)$ and $(1-p)q$ both vanish, yielding $p = q$.
\end{proof}

\section{Revisiting arbitrary discrete quantum groups}\label{se:dqg.nk}

\begin{theorem}\label{th:dqg.nk}
Let $M\cong \ell^{\infty}(\cX)$ be a hereditarily atomic von Neumann algebra, and let
  \begin{equation*}
    M\xrightarrow{\quad\Delta\quad}
    M\stensor M
    \quad\text{and}\quad
    M\xrightarrow{\quad\varepsilon\quad}\bC
  \end{equation*}  
  be $W^*$-morphisms such that
  \begin{equation*}
    (\Delta \stensor \id)\circ \Delta = (\id \stensor \Delta) \circ \Delta
    ,\quad
    (\varepsilon \stensor \id) \circ \Delta = \id,
    \quad\text{and}\quad
    (\id \stensor \varepsilon) \circ \Delta = \id.
  \end{equation*}
In other words, let $(M, \Delta, \varepsilon)$ be a discrete quantum monoid, as discussed in the Introduction.
  \begin{enumerate}[(1),wide]
  \item 
  If there exists a relation $\cX^*\xrightarrow{R}\cX$ such that
  \begin{enumerate}[(a),wide]
    \item\label{item:th:dqg.nk:sleft} $\hat \varepsilon^{\dag}\circ \hat \Delta\circ(R \times \id)
      \ge
      \ev_{\cX}$,
    \item\label{item:th:dqg.nk:sright} $\hat \varepsilon^{\dag}\circ \hat \Delta\circ(\id\times R)
      \ge
      \ev_{\cX^*}$,
    \end{enumerate}
    then $(M, \Delta, \varepsilon)$ is a discrete quantum group.
    
  \item Conversely, if $(M, \Delta, \varepsilon)$ is a discrete quantum group, then the relation $\cX^*\xrightarrow{R}\cX$ given by
  \begin{equation}\label{eq:Risshat}
      R\left(X_{\alpha}^*,X_{\beta}\right)
      :=
      \left\{
        T\in L\left(X_{\alpha}^*,X_{\beta}\right)
        \ :\
        Ts(m)
        =
        mT
        ,\quad
        \forall m\in M_{00}
      \right\}
  \end{equation}
 satisfies \Cref{item:th:dqg.nk:sleft} and \Cref{item:th:dqg.nk:sright} above, where $s\: M_{00} \to M_{00}^{op}$ is the antipode \cite[\S~4]{zbMATH00869786}.
  \end{enumerate}
\end{theorem}

\begin{remarks}\label{res:nonkac.cj}
  \begin{enumerate}[(1),wide]
  \item If $(M, \Delta, \varepsilon)$ is a discrete quantum group of Kac type, then its antipode extends to a $W^*$-morphism $M \to M^{op}$, and the relation $R$ in \Cref{eq:Risshat} is clearly $R = \hat s$, as in \Cref{re:intertw}.
  \item\label{item:res:nonkac.cj:nonkac.cj} \Cref{th:dqg.nk} provides a non-Kac version of \cite[p.17, Conjecture]{kor_qgr_sld}. The difference lies simply in requiring that the antipode $s$ be a \emph{relation} rather than \emph{function}.

  \item\label{item:res:nonkac.cj:proj.dom} Recall \Cref{not:rel.proj.fn}. We will have some use for the observation that for projections
    \begin{equation*}
      p\in \ell^{\infty}\left(\cY\right)
      \quad\text{and}\quad
      q\in \ell^{\infty}\left(\cX\right)\stensor \ell^{\infty}\left(\cY\right)
    \end{equation*}
    we have
    \begin{equation*}
      1\otimes p\ge q
      \xLeftrightarrow{\quad}
      p\ge P_{\hat q \circ (\top^{\cX} \times \id_\cY)},
    \end{equation*}
    where $p$ and $q$ are regarded as being associated to relations of type $\cY\to \mathbf{1}$ and $\cX \times \cY \to \mathbf{1}$, respectively \cite[Proposition~3.2.3]{zbMATH07828321}.
  \end{enumerate}
\end{remarks}

\pf{th:dqg.nk}
\begin{th:dqg.nk}
  \begin{enumerate}[(1),wide]
  \item
    Precomposing \Cref{item:th:dqg.nk:sleft} above with $\top^{\cX^*}\times\id$ produces 
    \begin{equation*}
     \hat \varepsilon^{\dag}\circ \hat \Delta\circ\left(\top^{\cX}\times \id\right)
      \ge
     \hat \varepsilon^{\dag}\circ \hat \Delta\circ\left(R \circ \top^{\cX^*}\times \id\right)
      \ge
      \ev_{\cX}\circ\left(\top^{\cX^*}\times\id\right)
      =
      \top_{\cX},
    \end{equation*}
    and similarly, precomposing \Cref{item:th:dqg.nk:sright} with $\id\times \top^{\cX^*}$ produces 
    \begin{equation*}
     \hat \varepsilon^{\dag}\circ \hat \Delta\circ\left(\id\times \top^{\cX}\right)
      \ge
      \ev_{\cX^*}\circ\left(\id\times \top^{\cX^*}\right)
      =
      \top_{\cX}.
    \end{equation*}
    These are precisely the two conditions imposed in \cite[p.16, Theorem]{kor_qgr_sld} and unpack (via \Cref{res:nonkac.cj}\Cref{item:res:nonkac.cj:proj.dom}) to the requirement that
    \begin{equation*}
      \forall\text{ projection }p\in M
      \quad:\quad
      (p\otimes 1) \ge \Delta(e) \text{ or }(1\otimes p)\ge \Delta(e)
      \xRightarrow{\quad}p=1,
    \end{equation*}
    where $e\in M$ is the support projection of $\varepsilon$. That this makes $(M,\Delta,\varepsilon)$ into a discrete quantum group is what the already-cited \cite{362309} shows.

  \item Conversely, consider a discrete quantum group $(M,\Delta,\varepsilon)$. We reprise the notation in the proof of \Cref{th:is.kac.bis}, setting $M\cong \ell^{\infty}(\cX)$ with $\At(\cX)=\{X_{\alpha}\}$. Being a quantum group, $(M,\Delta,\varepsilon)$ admits an anti-multiplicative antipode \cite[\S 4]{zbMATH00869786}
    \begin{equation*}
      M_{00}
      \xrightarrow{\quad s\quad}
      M_{00}
      \subseteq
      \mathrm{Mult}(M_{00})
      :=
      \prod_{X_{\alpha}\in \At(\cX)}L(X_{\alpha})
    \end{equation*}
    that maps each $L(X_{\alpha})$ bijectively onto $L(X_{\alpha^*})$ and that satisfies
    \begin{equation}\label{eq:what.antipodes.do}
       \mathrm{mult}\circ(\id\otimes s)\circ\Delta
        =
        \varepsilon
        =
        \mathrm{mult}\circ(s\otimes \id)\circ\Delta
      :
      M_{00}
      \xrightarrow{
        \quad
        \quad
      }
    \mathrm{Mult}(M_{00})
    \end{equation}
    for the multiplication map $\mathrm{mult}\: M_{00} \otimes M_{00} \to M_{00}$. The compositions in these equalities are interpreted via multipliers as described on \cite[p.438]{zbMATH00869786}.
    The proof of \Cref{th:is.kac.bis} can now effectively be run in reverse: the defining property \Cref{eq:what.antipodes.do} of the antipode is essentially just condition \Cref{eq:char.antp}. \qedhere
  \end{enumerate}
\end{th:dqg.nk}


\addcontentsline{toc}{section}{References}

\def\polhk#1{\setbox0=\hbox{#1}{\ooalign{\hidewidth
  \lower1.5ex\hbox{`}\hidewidth\crcr\unhbox0}}}

\Addresses

\end{document}